\documentclass[12pt]{amsart}

\usepackage{amsmath,amssymb,latexsym}
\usepackage{amsfonts,amstext,amsthm,amscd}
\usepackage{graphicx}
\usepackage{enumerate}
\usepackage{amsthm}
\usepackage{xypic}
\theoremstyle{plain}
\usepackage{color}

\theoremstyle{definition}
\newtheorem{theorem}{Theorem}[section]

\newtheorem{definition}[theorem]{Definition}
\newtheorem{example}[theorem]{Example}

\newtheorem{properties}[theorem]{Properties}
\newtheorem{proposition}[theorem]{Proposition}
\newtheorem{remark}[theorem]{Remark}
\numberwithin{equation}{section}

\newcommand{\norm}[1]{\left\Vert#1\right\Vert}
\newcommand{\abs}[1]{\left\vert#1\right\vert}

\newcommand{\A}{\mathbb A}
\newcommand{\Ann}{\mathrm{Ann}}
\newcommand{\AP}{\mathrm{C_{ap}}}

\newcommand{\C}{\mathbb{C}}
\newcommand{\Co}{\mathrm{C}}
\newcommand{\Char}{\mathrm{Char}}

\newcommand{\dR}{\Char(\R)}
\newcommand{\Hull}{\mathrm{Hull}}
\newcommand{\dZz}{\Char(\Zz)}
\newcommand{\Fphi}{\widehat{\Phi}}
\newcommand{\Sphi}{\overline{\Phi}}

\newcommand{\sphi}{\overline{\phi}}
\newcommand{\Hom}{\mathrm{Hom}}
\newcommand{\Lo}{\mathcal{L}_{0}}
\newcommand{\Ll}{\mathcal{L}}
\newcommand{\LP}{\mathrm{C_{lp}}}
\newcommand{\M}{\mathcal{M}}
\newcommand{\N}{\mathbb{N}}

\newcommand{\Orb}{\mathfrak O}
\newcommand{\Q}{\mathbb{Q}}
\newcommand{\R}{\mathbb{R}}
\newcommand{\Ss}{\mathsf{S}}
\newcommand{\UT}{\mathbb{T}}
\newcommand{\UC}{\mathbb{S}^{1}}
\newcommand{\Z}{\mathbb{Z}}
\newcommand{\Zz}{\widehat{\mathbb{Z}}}

\newcommand{\eps}{\epsilon}
\newcommand{\To}{\longrightarrow}
\newcommand{\mTo}{\longmapsto}

\date{\today}

\begin{document}

\title[Bohr -- Fourier Series on Solenoids]{Bohr -- Fourier Series on Solenoids via its Transversal Variation}
\author[M. Cruz -- L\'opez]{Manuel Cruz -- L\'opez*}
\address{*Departamento de Matem\'aticas
Universidad de Guanajuato, Jalisco s/n, Mineral de Valenciana,
Guanajuato, Gto. 36240 M\'exico.}
\email{manuelcl@ugto.mx}
\author[F. J. L\'opez -- Hern\'andez]{ Francisco J. L\'opez -- Hern\'andez**}
\address{**Instituto de F\'isica, 
Universidad Aut\'onoma de San Luis Potos\'i, 
Av. Manuel Nava No. 6, Zona Universitaria,
San Luis Potos\'i, SLP. 78290 M\'exico.}
\email{flopez@ifisica.uaslp.mx}

\begin{abstract}

The Bohr -- Fourier series development on one dimensional solenoids is analyzed by using invariant functions and extending Bohr's theory through the study of transversal variation. 
 
\end{abstract}

\subjclass[2010]{Primary: 22XX, 43XX, Secondary: 22Cxx, 43Axx}
\keywords{solenoids, Fourier series, tranversal variation}
\maketitle

\section[Introduction]{Introduction}
\label{introduction}

Since J. Fourier's introduction of his \emph{M\'emoire sur la propagation de la chaleur dans les corps solides} in 1807, the theory of what is now called Fourier Analysis shown its importance, not only from the purely theoretical point of view but also from the viewpoint of its applications. Many important mathematicians contributed greatly to the theory and in the mid 1930s the theory took a great impulse with the development of the theory of topological groups by A. Weil, L. Pontryagin, J. von Neumann, van Dantzing, among others. 

From there on, the theory is called \emph{Abstract Harmonic Analysis} and it has been developed particularly in several topological groups. The long complete description made by E. Hewitt and K.A. Ross in \cite{HR1,HR2} shows the importance of the subject.  

In this article, the Fourier analysis on one dimensional solenoids is done by following the path traced by H. Bohr in his celebrated theory of \emph{Almost periodic functions} (see \cite{Bohr}). By character theory, one dimensional solenoids are dual groups of additive subgroups of the group of rational numbers $\Q$ with the discrete topology, and hence they are homomorphic images of the so called \textsf{universal one dimensional solenoid} $\Ss$ which is the dual group of $\Q$. The group $\Ss$ is a compact abelian topological group and also has a structure of a one dimensional foliated space. 

By considering the properly discontinuously free action of $\Z$ on $\R\times \Zz$ given by
\[ \gamma\cdot (x,t) := (x+\gamma,t-\gamma) \quad (\gamma\in \Z), \]
the group $\Ss$ appears as the orbit space of this action, i.e. $\Ss=\R\times_{\Z} \Zz$. Here, $\Z$ is acting on $\R$ by covering transformations and on $\Zz$ by right translations. The group 
$\displaystyle{\Zz := \varprojlim_n \Z/n\Z}$ is the profinite completion of $\Z$, which is a compact abelian topological group, and also perfect and totally disconnected, and hence it is homeomorphic to the Cantor set. Being $\Zz$ the profinite completion of $\Z$, it admits a canonical inclusion of $\Z$ whose image is dense.

As a topological group, $\Ss$ is also isomorphic to the projective limit
\[ \Ss \cong \varprojlim_n \{ \UC, p_{nm} \} \]
with canonical projection $\Ss\To \UC$, determined by projection onto the first coordinate, which gives a locally trivial $\Zz$ -- bundle structure $\Zz\hookrightarrow \Ss \To \UC$. 

In the classical theory over the circle, it is very well known that there exists a one to one correspondance between the set 
\[ \{ \Z - \text{invariant functions } \R\To \C \} \] 
and  
\[ \{ \text{Continuous functions } \UC\To \C \}, \] 
via the universal covering projection $\pi:\R\To \UC$. By using the `covering' projection 
$\R\times \Zz\To \Ss$, it is established an analogous one to one correspondance between
\[ \{ \Z - \text{invariant functions } \R\times \Zz\To \C \} \]
and
\[ \{ \text{Continuous functions } \Ss\To \C \}. \]   

This is the starting point for the development of the theory in this article. Once this context is settled, the inspiration is Bohr's treatment. The notion of the mean value is introduced for this class of functions (see Section \ref{solenoidal_BohrFourier-series}): for any $\Z$ -- invariant functions 
$\Phi:\R\times \Zz\To \C$,
\[ \M(\Phi) := \lim_{T\to \infty}  \frac{1}{T} \int_{\Zz} \int_0^T \Phi(x,t) dx dt, \]
whenever this limit exists. Using this mean value, the Bohr -- Fourier transform of any such function 
$\Phi$ is: 
\[ 
\Fphi(\chi_{\lambda,\varrho}) := 
\M \big( \Phi(x,t) \overline{\chi_{\lambda,\varrho}(x,t)} \big), 
\] 
where $\chi_{\lambda,\varrho}=\chi_\lambda\cdot \chi_\varrho$ is any character of the product 
$\R\times \Zz$ identified with $\R\times \Q/\Z$ by duality. 

Now, to any transversal variable $t\in \Zz$ there corresponds a limit periodic function 
$\Phi_t:\R\To \C$. For $t=0$, the function $\Phi_0$ is the corresponding function on the base leaf 
$\Lo=\R\times \{0\}$. If $\Fphi_0(\chi_\lambda)=M(\Phi_0 \cdot\overline{\chi_\lambda})$ is the classical Bohr -- Fourier coefficient of $\Phi_0$, the continuous variation $\Zz\To \AP(\R)$ implies:
\smallskip

\textbf{Theorem \ref{Bohr-Fourier_coefficients}:}
$$
\Fphi(\chi_{\lambda,\varrho}) = 
\Fphi_0(\chi_\lambda) \cdot \int_{\Zz} A_{\lambda}(t)\cdot \overline{\chi_{\varrho}(t)} dt,
$$
where $A_{\lambda} : \Zz\To \UT$ defines a character on $\Zz$ determined by the transversal variation. 
\smallskip

The fact that characters are an orthonormal system provides the relation 
\[ \Fphi(\chi_{\lambda,\varrho}) = \Fphi_0(\chi_\lambda), \] 
when $\varrho=\lambda \mod \Z$, and $0$ in other case.
\smallskip

The Bohr -- Fourier series of $\Phi$ can now be written as:
$$ 
\Sphi(x,t) = 
\sum_{(\lambda,\varrho)\in \Omega_\Phi} 
\Fphi(\lambda,\varrho) \chi_{\lambda,\varrho}(x,t), 
$$
where $\Omega_\Phi$ is a countable subset of characters $\R\times \Q/\Z$. 

Denote by $\Co(\Ss)$ the set consisting of all $\Z$ -- invariant functions $\Phi:\R\times \Zz\To \C$.
The Parseval's identity is established as:
\smallskip

\noindent \textsf{Theorem \ref{solenoidal_parseval-identity}:}
For any $\Phi\in \Co(\Ss)$, 
\[ \sum_{(\lambda,\varrho)\in \Omega_\Phi} \abs{\Fphi(\lambda,\varrho)}^2 = \M(\abs{\Phi}^2). \]
\smallskip

The Approximation theorem follows:
\smallskip

\noindent \textbf{Theorem \ref{solenoidal_approximation-theorem}:}
Any $\Phi\in \Co(\Ss)$ can be approximated arbitrarily by finite terms of its Fourier series. 
\smallskip

It is important to point out that several attempts have been made to describe the theory of Fourier series on solenoids, most notably the work \cite{HRit}, where the authors used characters on solenoids composed with trigonometric polynomials on the circle to describe the theory. As has already been said, the approach in this article is to use the theory of $\Z$ -- invariant functions on the covering 
$\R\times \Zz$ which descend to the appropriate functions on $\Ss$ plus to follow the line of ideas of Bohr's treatment in order to be able to introduce the notion of mean value and the description of frequencies to form the Fourier series.

Further development of the theory presented here goes in two different directions: one the one hand, the full generalization to the $L^p$ theory would be able, and, on the other, the extension of these ideas to the so called Sullivan's solenoidal manifolds is propitious, since, according to Sullivan (see \cite{Sul} and \cite{Ver}), any compact one dimensional orientable solenoidal manifold is the suspension of a homeomorphism of the Cantor set. The universal solenoid itself is precisely one instance of this construction. All these themes are the subject of recent investigations. 

Section \ref{universal_solenoid} presents the relevant definitions on solenoids, characters, measures and all that. In Section \ref{Bohr_theory} there is a brief account of the most relevant facts to this article of the classical Bohr's theory. The Section \ref{solenoidal_theory} is dedicated to the description of basic ingredients of the solenoidal theory and in Section \ref{solenoidal_BohrFourier-series}, the Bohr -- Fourier series is described and compared with the classical Fourier series on $\Ss$, reminiscent of the series on an arbitrary compact abelian group. 

\section[The universal solenoid]{The universal solenoid}
\label{universal_solenoid}

This section introduces the basic objects relevant to this article: the universal solenoid exhibited as an orbit space, as a projective limit and also as a quotient group, the basic definitions and examples of duality theory, and the required elements of measure theory. A complete account of much of these concepts and properties is documented in the treatise \cite{HR1}. More specific descriptions of most of the objects presented here can be consulted in the recent article \cite{CLV}. 

\subsection[The universal solenoid]{The universal solenoid}
\label{solenoid}

For every integer $n\geq 1$, by covering space theory, it is defined the unbranched covering space of degree $n$, $p_n:\UC \To \UC$ given by $z\mTo z^n$. If $n,m\in \Z^+$ and $n$ divides $m$, then there exists a unique covering map 
$p_{nm}:\UC\To \UC$ such that $p_n \circ p_{nm} = p_m$.

This determines a projective system of covering spaces $\{\UC,p_n\}_{n\geq 1}$ whose projective limit is the \textsf{universal one dimensional solenoid}
\[ \Ss := \varprojlim_n \{ \UC, p_{nm} \} \]
with canonical projection $\Ss\To \UC$, determined by projection onto the first coordinate, which produces a locally trivial $\Zz$ -- bundle structure 
$\Zz\hookrightarrow \Ss \To \UC$, where $\Zz := \displaystyle{\varprojlim_n \Z/m\Z}$ is the profinite completion of $\Z$, which is a compact, perfect and totally disconnected abelian topological group homeomorphic to the Cantor set. The image of the inclusion 
$\Z\hookrightarrow \Zz$ is dense.

By considering the properly discontinuously free action of $\Z$ on $\R\times \Zz$  given by
\[ \gamma\cdot (x,t) := (x+\gamma,t-\gamma) \quad (\gamma\in \Z), \]
$\Ss$ is identified with the orbit space $\R\times_{\Z} \Zz \equiv \R\times \Zz / \Z$. Here, $\Z$ is acting on $\R$ by covering transformations and on $\Zz$ by translations. The pathconnected component of the identity element $0\in \Ss$ is called the \textsf{base leaf} and it is denoted by $\Lo$. Clearly, $\Lo$ is the image of $\R\times \{0\}$ under the canonical projection $\R\times \Zz\To \Ss$ and it is homeomorphic to $\R$.
\smallskip

In summary, $\Ss$ is a compact, connected, abelian topological group and also a one dimensional lamination where each ``leaf" is a simply connected one dimensional manifold, homeomorphic to the universal covering space $\R$ of $\UC$ and a typical ``transversal" is isomorphic to the Cantor group 
$\Zz$.

\subsection[Characters]{Characters}
\label{characters}

The \textsf{group of characters} or \textsf{dual group} of $\R$ is the group consisting of continuous homomorphims $\Hom_{\mathrm{cont}}(\R,\UC)$ denoted by $\dR$, and similarly define the character group of any abelian group. By the classical theory, the group of characters of a compact abelian group is a discrete abelian group, and viceversa, the character group of a discrete abelian group is a compact abelian group. Also, the character group of a product of two abelian groups is the product of the character groups. The following examples and facts are relevant for this work and they are very well known:
\begin{enumerate}[(a)]
\item $\dR \cong \R$,
\item $\dZz\cong \Q/\Z$, where $\Q/\Z$ is the group of roots of unity,
\item $\Char(\R\times \Zz)\cong \Char(\R)\times \Char(\Zz)\cong \R\times \Q/\Z$.
\end{enumerate}

The statement (c) says that any character in $\R\times \Zz$ has the form
\[ \chi_{\lambda,\varrho} = \chi_\lambda \cdot \chi_\varrho, \]
for some $\lambda\in \R$ and $\varrho\in \Q/\Z$.

An important character group for this development is
\begin{remark}
$\Char(\Ss)\cong \Q$.
\end{remark}

Classically this isomorphism is deduced from the fact that there is an isomorphism of topological groups between the solenoid $\Ss$ and the so called \emph{Ad\`ele Class Group} of the rational numbers 
$\A_\Q/\Q$, where $\A_\Q$ is the ad\`ele group of $\Q$ and $\Q\hookrightarrow \A$ is a discrete cocompact subgroup. However, for the purposes of this article it is convenient to calculate the character group of $\Ss$ in an alternative way as follows.

The solenoid $\Ss$ can also be realized as the quotient group $\R\times \Zz/\Z$, where 
$\Z$ is immersed diagonally as a discrete subgroup by
\[ \Z\hookrightarrow \R\times \Zz, \qquad n\mTo (-n,n). \]

In order to be able to compute the dual group of a quotient group, the duality theory establish an isomorphism
\[ \Char(\R\times \Zz/\Z) \cong \Ann(\Z), \]
where $\Ann(\Z)$ is the annihilator subgroup of $\Z$ in $\Char(\R\times \Zz)$. It happens that the characters in $\R\times \Zz$ which annihilates the generator $(-1,1)$ of $\Z$ in the product are precisely the characters determined by elements in $\Z\times \Q/\Z$.

By duality theory, the surjective homomorphism $\R\times \Zz\To \Ss$ induces a monomorphism between the dual groups $\Q\To \R\times \Zz$ whose image is isomorphic to the subgroup 
$\Z\times \Q/\Z$. This identification is very important in this work:

\begin{remark}
\label{frequencies_up-down}
There is a one to one correspondance between discrete abelian groups $\Q$ and $\Z\times \Q/\Z$.
\end{remark}
 
\subsection[Haar measure]{Haar measure}
\label{haar_measure}

Denote by $dx$ the usual Haar measure on $\R$ and by $dt$ the Haar measure on $\Zz$ normalized in such a way that
\[ \int_{\Zz} dt = 1. \]

So, the Haar measure on $\R\times \Zz$ is the product measure $dx\times dt=dxdt$ and it induces the normalized Haar measure $d\mu$ on $\Ss$, i.e.
\[ \int_{\Ss}\phi d\mu= \int_{\Zz}\int_{\R} \Phi  dxdt,  \]
for any lifting $\Phi:\R\times \Zz\To \C$ of $\phi:\Ss\To \C$.

\section[Classical Bohr's theory]{Classical Bohr's theory}
\label{Bohr_theory}

This section is a brief r\'esum\'e of Bohr's theory of almost periodic functions. We follow closely Bohr's seminal work \cite{Bohr}.
\smallskip

Let $\Co(\R)$ be the space of complex valued continuous functions equipped with the uniform norm. Define the action by translations of $\R$ on $\Co(\R)$ by 
\[ \R\times \Co(\R)\to \Co(\R),\quad (t,\varphi)\mTo \varphi^t = \varphi\circ R_t, \]
where $\varphi^t:\R\To \R$ is given by $\varphi^t(x) := \varphi\circ R_t(x) = \varphi(x+t)$. 

Denote by $\Orb_{\R}(\varphi)$ the orbit of $\varphi$ under this action, and by $\Hull(\varphi)$ the  closed convex hull of $\Orb_{\R}(\varphi)$ in $\Co(\R)$.

Given $\varphi\in \Co(\R)$ and $\eps>0$, the number $\tau=\tau(\eps)\in \R$ is called a \textsf{translation number} of $\varphi$ (corresponding to $\epsilon$) whenever
\[ \norm{\varphi^{\tau(\epsilon)} - \varphi}_{\infty}\leq\eps. \]

\begin{definition}
$\varphi\in \Co(\R)$ is called \textsf{almost periodic} if given $\eps>0$, there exists a relatively dense set of translation numbers of $\varphi$ corresponding to $\eps$, i.e. for all $\eps$, there exists a length $L=L(\eps)$ such that each interval of length $L$ contains at least one translation number corresponding to $\eps$. 
\end{definition}

Denote by $\AP(\R)$ the complex vector space consisting of all almost periodic functions.

\begin{example}
Any periodic function is obviously an almost periodic function.
\end{example}

Some important properties of almost periodic functions are summarized in the following:

\begin{properties}
The following properties are satisfied:
\begin{enumerate}
\item If $\varphi\in \AP(\R)$, then $\varphi$ is an uniformly continuous function.
\item The sum of almost periodic functions is an almost periodic function.
\item The uniform limit of almost periodic functions is an almost periodic function.
\end{enumerate}
\end{properties}

Since the sum of arbitrary periodic functions is an almost periodic function, particularly the trigonometric polynomials are almost periodic functions. An interesting observation is that  every function $\varphi$ which can be approximated uniformly by trigonometric polynomials is an almost periodic function (see Theorem \ref{Bohr_theorem}). 

The main interest in this article is the subspace of all limit periodic functions 
$\LP(\R)\subset \AP(\R)$ which consists of all functions $\varphi$, such that $\varphi$ is the uniform limit of periodic functions.

\begin{definition}
For every almost periodic function there exists the \textsf{mean value}
\[ M(\varphi) = \lim_{T\to \infty} \frac{1}{T} \int_{0}^{T} \varphi(x) dx. \]
\end{definition}

It is clear that if $\varphi\in \AP(\R)$ and $t\in\R$, then $\varphi^t \in \AP(\R)$, and therefore there exists $M(\varphi^t)$.

\begin{theorem}
\label{Mean_properties}
$M : \AP(\R) \To \C$ is a continuous linear functional which is invariant under right translations. That is,
\begin{enumerate}
\item $M(\varphi+\psi) = M(\varphi) + M(\psi)$, for any $\varphi,\psi\in \AP(\R)$.
\item $M(\varphi^t) = M(\varphi)$, for any $\varphi\in \AP(\R)$ and $t\in \R$.
\item If $\varphi$ is the uniform limit of a sequence $(\varphi_n)_{n\in \N}$, then
\[ M(\varphi)=\lim_{n\to \infty} M(\varphi_n). \]
\end{enumerate}
\end{theorem}

Now recall the concept of the Fourier series of an almost periodic function. A \textsf{normalized orthogonal system}
$\{ e^{i\lambda x} \}_{\lambda\in\R}$ satisfies
\[ M(e^{i\lambda_1 x} e^{-i\lambda_2 x}) = \delta(\lambda_1,\lambda_2) \]
where $\delta(\lambda_1,\lambda_2)=1$ if $\lambda_1=\lambda_2$ and 0 in other case. The elements of this system are called \textsf{basic elements} and this set can be identified with $\Char(\R)$.
\smallskip

Consider $\varphi\in \AP(\R)$ and $\lambda\in \R$. The function 
$\varphi(x) e^{-i\lambda x}$ is the product of an almost periodic function and a purely periodic function, so it is an almost periodic function and its mean value
$$
M(\varphi(x) e^{-i\lambda x}) = 
\lim_{T\to \infty} \frac{1}{T} \int_{0}^{T} \varphi(x) e^{-i\lambda x} dx
$$
exists.

The next theorem is of fundamental importance for the theory.

\begin{theorem}
\label{Bohr_frequencies}
The function $a(\lambda) := M(\varphi(x) e^{-i\lambda x})$ is zero for all values of 
$\lambda$ with the exception of at most an enumerable set of numbers $\lambda$.
\end{theorem}

This theorem allows to carry the theory of Bohr -- Fourier series into the theory of almost periodic functions in the sense that it is possible to associate to an almost periodic function 
$\varphi$ its unique Bohr -- Fourier series
\[ \sum_{n\in \N} a(\lambda_n) e^{i\lambda_n x}. \]

\begin{remark} 
\label{Parseval}
The Parseval's identity holds for any almost periodic function $\varphi$:
\[ \sum_{n\in \N} \abs{a(\lambda_n)}^2 = M(\abs{\varphi}^2). \]
\end{remark}

The main result of the theory goes as follows.

\begin{theorem} \textsf{(Bohr)}
\label{Bohr_theorem}
Every almost periodic function can be uniformly approximated by finite sums
$s_N(x)=\sum_1^N a_n e^{i\lambda_n x}$. The exponents in the approximating sums $s_N(x)$ can be chosen to be precisely the Fourier exponents $\lambda_n$ of the function 
$\varphi$.
\end{theorem}
 
\section[Solenoidal Bohr -- Fourier theory]{Solenoidal Bohr -- Fourier theory}
\label{solenoidal_theory}

This section presents the main basic elements required for the development of the theory of the solenoidal Bohr -- Fourier series. First, it will be analyzed the relevant spaces of continuous functions, both on $\Ss$ and on $\R\times \Zz$, and the continuous variation of the functions with respect to the transversal variable. This allows to define the appropriate notion of mean value and to describe its transversal variation.   

\subsection[Continuous invariant functions on $\Ss$]{Continuous invariant functions on $\Ss$}
\label{continuous_functions}
 
Denote by $\LP(\R)$ the space of limit periodic functions $\R\To \C$ in the sense of Bohr. Let 
$\Co(\Ss)$ be the space of continuous functions $\phi:\Ss\To\C$. It is well known that there is a one to one correspondance between $\Co(\Ss)$ and the space $\Co_\Z(\R\times \Zz)$ of continuous function $\Phi:\R\times \Zz\To \C$ satisfying that $\Phi$ is invariant under the action of $\Z$, i.e.
\[ 
\Phi(\gamma\cdot (x,t)) = \Phi(x+\gamma,t-\gamma) = \Phi(x,t), \qquad
((x,t)\in \R\times \Zz, \gamma\in \Z).
\]
 
In order to develop the  Bohr -- Fourier theory for $\Co(\Ss)$ we will work on the space 
$\Co_\Z(\R\times \Zz)$, which, after projection provides us the Bohr -- Fourier theory of 
 $\Co(\Ss)$ described at the end of Section \ref{solenoidal_BohrFourier-series}. For now on, we will indistinguishably denote by  $\Co(\Ss)$ both spaces.

For every $t\in \Zz$, the function $\Phi_t:\R\To \C$ defined by
\[ \Phi_t(x) = \Phi(x,t) \]
is continuous. The invariant condition can be written as
\begin{equation}
\label{invariant_condition}
\Phi_{t-\gamma} (x + \gamma) = \Phi_t(x), \qquad ((x,t)\in \R\times \Zz, \gamma \in \Z).
\end{equation}

\begin{remark}
According to \cite{Lop}, Theorem 2.4, for every $t\in \Zz$, the function $\Phi_t:\R\To \C$ is limit periodic.
\end{remark}

A nice consequence of this remark is the following:

\begin{proposition}
For each $\Phi\in \Co(\Ss)$, the map
\[ \Zz\To \LP(\R), \qquad t\To \Phi_t \]
is uniformly continuous. That is, if $(t_n)_{n\geq 1}$ is a sequence of points in 
$\Z\subset \Zz$ which converges to $t\in \Zz$ in the profinite topology, then the sequence $(\Phi_{t_n})_{n\geq 1}$ in $\LP(\R)$ converges to $\Phi_t\in \LP(\R)$ in the uniform topology of $\LP(\R)$.
\end{proposition}

This Proposition implies 
\[ \Co(\Ss) \cong \Co(\Zz,\LP(\R)). \]

\begin{remark}
\label{translation-variation}
As a matter of notation, it is important to notice that the function $\Phi_t$ \textbf{does not correspond} exactly with the usual definition of right translations on $\Co(\R)$ which is denoted by $\Phi^t$ with $t\in \R$. This notation emphazises the dependence of $\Phi$ on the transversal variable. However, when $t=0$ in $\Zz$, the invariant condition restricted to 
$\Ll_0$ implies that
\begin{align*}
\Phi_0^s(x)	&= \Phi_0\circ R_s(x) \\
					&= \Phi_0(x+s) \\
					&= \Phi(x+s,0) \\
					&= \Phi(x+s+(-s),0-(-s)) \\
					&= \Phi(x,s)\\
                   &= \Phi_s(x),
\end{align*}
for any $s\in \Z$ and $x\in \Ll_0$. Furthermore, for any $t,s\in \Z\subset \Zz$ and 
$x\in \Ll_t$, the relation above sees as:
\begin{align*}
\Phi_t^s(x)	&= \Phi_t\circ R_s(x) \\
					&= \Phi_t(x+s) \\
					&= \Phi(x+s,t) \\
					&= \Phi(x+s+(-s),t-(-s)) \\
					&= \Phi(x,t+s)\\
                   &= \Phi_{t+s}(x).
\end{align*}

\end{remark}

\subsection[The mean value]{The mean value}
\label{solenoidal_mean-value}

For any function $\Phi\in \Co(\Ss)$, the \textsf{mean value} of $\Phi$ is given by
\[ \M(\Phi) = \lim_{T\to \infty}  \frac{1}{T} \int_{\Zz} \int_0^T \Phi(x,t) dx dt, \]
whenever this limit exists.

\begin{theorem}
\label{MeanComparison}
$\M(\Phi) = M(\Phi_t)$, for any choice of $t\in \Zz$ fixed.
\end{theorem}

\begin{proof}
If $t\in \Z$ and $s\in\Z$, Remark \ref{translation-variation} implies that 
$\Phi_{t+s}=\Phi_t\circ R_s$. By traslation invariance of the mean value (see Theorem \ref{Mean_properties}(2)), it follows that
\[ M(\Phi_{t+s}) = M(\Phi_t\circ R_s) = M(\Phi_t) \quad (t\in \Z). \] 

Now, by what have been said before, if $(t_n)_{n\geq 1}$ is a sequence of points in 
$\Z\subset \Zz$ which converges to $t\in \Zz$ in the profinite topology, then the sequence $(\Phi_{t_n})_{n\geq 1}$ converges to $\Phi_t$. By properties of the mean value (see Theorem \ref{Mean_properties}(3)), 
$\displaystyle{M(\Phi_t)=\lim_{n\to \infty} M(\Phi_{t_n})}$. This means that for any 
$t\in \Zz$ fixed, the mean value is constant and equal to $M(\Phi_t)$ in $\Zz$. Therefore
\begin{align*}
\M(\Phi)	&= \lim_{T\to \infty}  \frac{1}{T} \int_{\Zz} \int_0^T \Phi(x,t) dx dt \\
				&= \int_{\Zz}  M(\Phi_t) dt \\
				&= M(\Phi_t).
\end{align*}
\end{proof}

\begin{theorem}
The invariant mean $\M : \Co(\Ss)\To \C$ is a continuous linear functional which is invariant under right translations. That is,
\begin{enumerate}
\item $\M(\Phi+\Psi) = \M(\Phi) + \M(\Psi)$, for any $\Phi,\Psi\in \Co(\Ss)$.
\item $\M(\Phi\circ R_s) = \M(\Phi)$, for any $\Phi\in \Co(\Ss)$ and $s\in \R$.
\item If $\Phi$ is the uniform limit of a sequence $(\Phi_n)_{n\in \N}$, then
\[ \M(\Phi) = \lim_{n\to \infty} \M(\Phi_n). \]
\end{enumerate}
\end{theorem}

\subsection[Bohr -- Fourier transform]{Bohr -- Fourier transform}
\label{solenoidal_Bohr-Fourier_transform}

Given any function $\Phi\in \Co(\Ss)$ and any character 
$\chi_{\lambda,\varrho} \in \Char(\R\times \Zz)$, the Fourier transform of $\Phi$ in the \textsf{mean sense} is given by
$$ 
\Fphi(\chi_{\lambda,\varrho}) = 
\M \big( \Phi(x,t) \overline{\chi_{\lambda,\varrho}(x,t)} \big) = 
\lim_{T\to \infty} \frac{1}{T} \int_{\Zz} \int_0^T \Phi(x,t) \overline{\chi_{\lambda,\varrho}(x,t)} dx dt.
$$

In fact,
\begin{theorem}
\label{transform_expression}
If $\Phi$ is any function in $\Co(\Ss)$ and $\chi_{\lambda,\varrho}$ is any element in 
$\Char(\R\times \Zz)$, then
\[ 
\Fphi(\chi_{\lambda,\varrho}) = \int_{\Zz} M(\Phi_t e^{-i\lambda x})\overline{\chi_{\varrho}(t)} dt. 
\]
\end{theorem}

\begin{proof}
\begin{align*}
\Fphi(\chi_{\lambda,\varrho})	
&= \lim_{T\to \infty}  \frac{1}{T} \int_{\Zz} \int_0^T \Phi(x,t) \overline{\chi_{\lambda,\varrho}(x,t)} dx dt \\
&= \lim_{T\to \infty}  \frac{1}{T} \int_{\Zz} \int_0^T \Phi(x,t) \overline{\chi_{\lambda}(x)} \overline{\chi_{\varrho}(t)} dx dt \\
&= \lim_{T\to \infty}  \frac{1}{T} \int_{\Zz} \int_0^T \Phi(x,t) e^{-i\lambda x}
\overline{\chi_{\varrho}(t)} dx dt \\
&= \int_{\Zz} \lim_{T\to \infty}\frac{1}{T}\int_0^T \Phi_t(x) e^{-i\lambda x} dx \cdot\overline{\chi_{\varrho}(t)} dt \\
&= \int_{\Zz} M(\Phi_t e^{-i\lambda x})\overline{\chi_{\varrho}(t)} dt.
\end{align*}
\end{proof}

Since $\Phi_t$ is limit periodic for all $t\in \Zz$, $\Hull(\Phi_t)$ is a quotient group of the solenoid (see \cite{Lop}, Theorem 2.2). By duality, $\Char(\Hull(\Phi_t))$ is a subgroup of the group 
$\Char(\Ss)\cong \Q$. 

\begin{remark}
\label{rational_frequencies}
The function $M(\Phi_{t} e^{-i\lambda x})$ is zero for all values of $\lambda$ with the exception of at most an enumerable subset $\Omega_{\Phi_t}$ of $\Q$.
\end{remark}

Theorem \ref{transform_expression} tells us that the study of the variation of 
$M(\Phi_t e^{-i\lambda x})$ with respect to the transversal variable $t$ must be done. The following discussion deals with this issue.
\smallskip

First, fix $t=0$, the identity element in $\Zz$. The function $\Phi_0\in \LP(\R)$ is a limit periodic function defined on the base leaf $\Lo = \R\times \{0\}\subset \R\times \Zz$. According to Bohr's theory: 
\begin{enumerate}[$\bullet$]
\item The frequency module of $\Phi_0$ is a countable subset of rational numbers 
$\Omega_{\Phi_0}\subset \R$, 
\item the invariant mean $M(\Phi_0)$ defined as
\[ M(\Phi_0) = \lim_{T\to \infty} \frac{1}{T} \int_0^T \Phi_0(x) dx \]
exists, and,
\item the $\lambda^{th}$ Fourier coefficient of $\Phi_0$,
$$
\Fphi_0(\lambda) = M(\Phi_0(x)e^{-i\lambda x}) =
\lim_{T\to \infty} \frac{1}{T} \int_0^T \Phi_0(x) e^{-i\lambda x} dx 
$$
is well defined. 
\end{enumerate}

\begin{remark}
Using the fact that $\R$ is selfdual, sometimes we will also write
$$ 
\Fphi_0(\chi_\lambda) = M(\Phi_0(x) \overline{\chi_\lambda(x)}) =
\lim_{T\to \infty} \frac{1}{T} \int_0^T \Phi_0(x) \overline{\chi_\lambda(x)} dx, 
$$
emphasizing the use of the character $\chi_\lambda \in \dR$ associated with $\lambda$. 
\end{remark}

The Fourier series of $\Phi_0$ is written as
\[ \Phi_0(x) = \sum_{\lambda \in \Omega_{\Phi_0}} \Fphi_0(\lambda) \chi_\lambda(x) =
\sum_{\lambda \in \Omega_{\Phi_0}} \Fphi_0(\lambda) e^{i\lambda x}. \]

\begin{theorem}
\label{transversal-variation}
If $\Phi\in\Co(\Ss)$ then  
\[ M(\Phi_t e^{-i\lambda x}) = A_{\lambda}(t) M(\Phi_0(x)e^{-i\lambda x}), \]
where $A_{\lambda} : \Zz\To \UT$ is a continuous function.
\end{theorem}

\begin{proof}
The first part of Remark \ref{translation-variation} implies that for any $t\in \Z\subset \Zz$, the identity $\Phi_0(x+t) = \Phi_t(x)$ holds for every $x\in \Lo$. The mean value of 
$\Phi_0$ is precisely invariant under these translations, i.e. $M(\Phi_0(x+t)) = M(\Phi_0(x))$. Hence,
\begin{align*}
M(\Phi_t(x)e^{-i\lambda x})
&= \lim_{T\to \infty} \frac{1}{T} \int_0^T \Phi_t(x) e^{-i\lambda x} dx \\
&= \lim_{T\to \infty} \frac{1}{T} \int_0^T \Phi_0(x+t) e^{-i\lambda x} dx \\
&= e^{-i\lambda t} \lim_{T\to \infty}\frac{1}{T} \int_0^T \Phi_0(x) e^{-i\lambda x}dx\\
&= e^{-i\lambda t} M(\Phi_0(x)e^{-i\lambda x}).
\end{align*}

This means that for any $t\in \Z\subset \Zz$, the mean value $M(\Phi_t(x)e^{-i\lambda x})$ is transformed into $e^{-i\lambda t} M(\Phi_0(x)e^{-i\lambda x})$. This calculation, together with the continuous variation can be used to determine the mean value $M(\Phi_t(x)e^{-i\lambda x})$ for any $t\in \Zz$. Chose a sequence $(t_{n})_{n\in \N}$ in 
$\Z\subset \Zz$ such that $t_n\to t$. Note that since $\Phi_{t_n}\To \Phi_t$,
\begin{align*}
M(\Phi_t(x)e^{-i\lambda x})	
&= \lim_{T\to \infty} \frac{1}{T} \int_0^T \Phi_t(x) e^{-i\lambda x} dx \\
&=\lim_{n\to \infty}\lim_{T\to \infty} \frac{1}{T} \int_0^T \Phi_{t_n}(x) e^{-i\lambda x} dx \\
&=\lim_{n\to \infty}\lim_{T\to \infty} \frac{1}{T} \int_0^T \Phi_0(x+t_n) e^{-i\lambda x} dx \\
&= \lim_{n\to \infty} e^{-i\lambda t_n} \lim_{T\to \infty} \frac{1}{T} 
\int_0^T \Phi_0(x) e^{-i\lambda x}dx\\
&= \lim_{n\to \infty} e^{-i\lambda t_{n}} M(\Phi_0(x)e^{-i\lambda x}).\\
&=  A_{\lambda}(t) M(\Phi_0(x)e^{-i\lambda x}),
\end{align*}
where
\[ A_{\lambda}(t) := \lim_{n\to \infty} e^{-i\lambda t_n} \]
exists and it does not depend on the choice of the sequence $(t_n)_{n\in \N}$. This determines a continuous function $A_{\lambda} : \Zz\To \UT$.
\end{proof}

\begin{remark}
\label{A_character} 
 Note that $A_{\lambda}$ can be written as
$$
A_{\lambda}(t) = \frac{M(\Phi_{t}(x)e^{-i\lambda x})}{M(\Phi_0(x)e^{-i\lambda x})} 
						= \frac{\Fphi_t(\chi_{\lambda}) }{ \Fphi_0(\chi_\lambda)}.
$$
The function $A_{\lambda} : \Zz\To \UT$ defines a character on $\Zz$.
\end{remark}

The results proved before, Theorem \ref{transform_expression}, Theorem 
\ref{transversal-variation} and Remark \ref{A_character}, can be used to compute the Fourier transform of any function $\Phi\in \Co(\Ss)$ in the following way: for any character 
$\chi_{\lambda,\varrho}\in \Char(\R\times \Zz)$,
\begin{align*}
\Fphi(\chi_{\lambda,\varrho}) &= 
\M \big( \Phi(x,t) \overline{\chi_{\lambda,\varrho}(x,t)} \big) \\
&= \lim_{T\to \infty} \frac{1}{T}  \int_{\Zz} \int_0^T
\Phi(x,t) \overline{\chi_{\lambda,\varrho}(x,t)} dx dt \\
&= \int_{\Zz} M(\Phi_t e^{-i\lambda x})\overline{\chi_{\varrho}(t)} dt \\
&= M(\Phi_0 \cdot\overline{\chi_\lambda}) \cdot \int_{\Zz} A_{\lambda}(t)\cdot \overline{\chi_{\varrho}(t)}.
\end{align*}

According to Remark \ref{rational_frequencies}, the mean value 
$M(\Phi_0 \cdot\overline{\chi_\lambda})$ is zero for all values of $\lambda$ with the exception of at most a countable subset of $\Q$.

The integral expression in the last equality is evaluated by integration of characters of 
$\Zz$:  
\[ \int_{\Zz} A_{\lambda}(t)\cdot \overline{\chi_{\varrho}(t)} dt = 1 \] 
if and only if $A_{\lambda}(t)=\chi_{\varrho}(t)$ and $0$ in other case. Also, $A_{\lambda}$ and $\chi_{\varrho}$ define the same character if and only if 
$\varrho=\lambda \mod \Z$.

So, the final form of the Fourier coefficient of any function $\Phi\in \Co(\Ss)$ is given in the following:

\begin{theorem}
\label{Bohr-Fourier_coefficients}
$$
\Fphi(\chi_{\lambda,\varrho}) = \Fphi_0(\chi_\lambda) \cdot 
\int_{\Zz} A_{\lambda}(t)\cdot \overline{\chi_{\varrho}(t)} dt,
$$
where $\Fphi_0(\chi_\lambda)=M(\Phi_0 \cdot\overline{\chi_\lambda})$ when 
$\varrho=\lambda \mod \Z$, and $0$ in other case.
\end{theorem}

\begin{remark}
\label{frequencies_decomposition}
As a consequence of the above theorem the mean value is zero except at most in an enumerable set $\Omega_{\Phi}\cong \Omega_{\Phi_0}$ (Compare Theorem \ref{Bohr_frequencies}). In fact, any $\lambda\in \Omega_{\Phi}$ can be written as 
$\lambda = [\lambda] + \varrho$, where $[\lambda]$ is the integer part of $\lambda$ and its fractional part $\varrho$ is such that $\varrho=\lambda \mod \Z$.
\end{remark}

\section[Solenoidal Bohr -- Fourier series]{Solenoidal Bohr -- Fourier series}
\label{solenoidal_BohrFourier-series}

This final section describes the Bohr -- Fourier series for a complex -- valued continuous function $\phi$ on the solenoid $\Ss$ through the associated continuous $\Z$ -- invariant function on $\R\times \Zz$. It also presents the solenoidal version of the Parseval's identity and the Approximation theorem. Finally, this theory is compared with the classical theory on $\Ss$ viewed as a compact abelian group.  

\subsection[Classical Fourier series on $\Ss$]{Classical Fourier series on $\Ss$}
\label{fourier-series_S}

According to the classical harmonic analysis on the compact abelian topological group 
$\Ss$, given a function $\phi:\Ss\To\C$, the Fourier series can be defined abstractly as 
\[ \sphi(z) = \sum_{q\in\Q} \widehat{\phi}(\chi_q) \chi_q(z), \]
where $\chi_q$ is the character of $\Ss$ associated to $q\in \Q$ and  
\[ \widehat{\phi}(\chi_q) = \int_{\Ss} \phi(z) \overline{\chi}_q(z) d\mu. \] 

In what follows the corresponding Bohr -- Fourier series described through the theory developed previously is done.
 
\subsection[Solenoidal Bohr -- Fourier series]{Solenoidal Bohr -- Fourier series}
\label{solenoidal_BF-series}

Denote by $\Sphi$ the Bohr -- Fourier series associated to a given function 
$\Phi\in \Co(\Ss)$.  According with Remark \ref{frequencies_decomposition}, the Bohr -- Fourier series of $\Phi$ is
$$ 
\Sphi(x,t) = 
\sum_{(\lambda,\varrho)\in \Omega_\Phi} 
\Fphi(\lambda,\varrho) \chi_{\lambda,\varrho}(x,t). 
$$

\begin{remark}
Since $\chi_{\lambda,\varrho}=\chi_\lambda\cdot \chi_\varrho$, when $t=0$, $\chi_\varrho(0)=1$ for any $\varrho$. By Theorem \ref{Bohr-Fourier_coefficients}, $\Fphi(\lambda,\varrho)=\Fphi_0(\lambda)$ and Remark \ref{frequencies_decomposition} shows that $\Omega_{\Phi}\cong \Omega_{\Phi_0}$. This allows to identify the Fourier series introduced here with the usual Bohr -- Fourier series when restricting to the base leaf $\Lo$:
$$ 
\Sphi_0(x) = \sum_{\lambda\in \Omega_{\Phi_0}} \Fphi_0(\lambda) \chi_\lambda (x). 
$$ 
\end{remark}

Following the order of ideas presented by Bohr (see \cite{Bohr}, Sections 70 and 84), the solenoidal version of the main results of Bohr's theory such as the Parseval's identity, the uniqueness theorem and the approximation theorem are now discussed.

\begin{theorem}[Parseval's identity]
\label{solenoidal_parseval-identity}
For any $\Phi\in \Co(\Ss)$ 
\[ \sum_{(\lambda,\varrho)\in \Omega_\Phi} 
\abs{\Fphi(\lambda,\varrho)}^2 = \M(\abs{\Phi}^2). \]
\end{theorem}

\begin{proof}
According with Remark \ref{frequencies_decomposition} and Theorem 
\ref{Bohr-Fourier_coefficients}, 
\[ \abs{\M(\Phi(\lambda,\varrho))}^2 = \abs{M(\Phi_0(\lambda))}^2. \]

Therefore, considering the Bohr -- Fourier series of $\Phi_0$, the classical Parseval's identity (see Theorem \ref{Parseval}) and Theorem \ref{MeanComparison} imply that 
\begin{align*}
\sum_{(\lambda,\varrho)\in \Omega_\Phi} \abs{\Fphi(\lambda,\varrho)}^2 &= \sum_{\lambda\in \Omega_{\Phi_0}} \abs{\Fphi_0(\lambda)}^2\\
&= M(\abs{\Phi_0}^2)\\
&= \M(\abs{\Phi}^2).
\end{align*}  
\end{proof}

\begin{theorem}[Uniqueness]
Any  $\Phi\in \Co(\Ss)$ is uniquely determined by its Fourier series.
\end{theorem}

\begin{proof}
Uniqueness follows from Parseval's identity as in Bohr (see \cite{Bohr} Section 71).  By Theorem \ref{MeanComparison}, if $\M(\abs{\Phi}^2)=0=M(\abs{\Phi_0}^2)$ then the second equality implies 
that $\Phi_0=0$. Finally, $\Phi_0\equiv 0$ implies $\Phi_t\equiv 0$ for every $t\in \Zz$ and therefore $\Phi\equiv 0$.
\end{proof}

\begin{remark}
As was established by Bohr, these theorems are equivalent and play a fundamental role in the development of the theory.
\end{remark}

Another implication of the theory developed here is that since any function can be approximated on the base leaf by the Fourier series in the classical sense and it coincides with the restriction of the solenoidal version, we can extend the argument to the solenoid by limits and the approximation theorem follows immediately.

\begin{theorem} [Approximation theorem]
\label{solenoidal_approximation-theorem}
Any $\Phi\in \Co(\Ss)$ can be aproximated arbitrarily by finite terms of its Fourier series. 
\end{theorem}

\subsection[Invariance of the Bohr -- Fourier series]{Invariance of the Bohr -- Fourier series}
\label{invariance_BohrFourier-series}

To conclude the analysis on the Bohr -- Fourier series it should be verify that the theory just developed descend naturally to the universal solenoid. This is done as follows. First recall that the invariance of any 
$\Phi\in \Co(\Ss)$ under the action of $\Z$, reads as:
$$\Phi_{t-\gamma} (x + \gamma) = \Phi_t(x), \qquad ((x,t)\in \R\times \Zz, \gamma \in \Z)$$

 From this expression and the definition of the Fourier coefficient follows immediately that the Bohr -- Fourier coefficients are invariant under the action of $\Z$. Hence the corresponding Fourier coefficients of the induced function $\phi$ are given by (see Section \ref{haar_measure})
\begin{align*}
\Fphi(\chi_{\lambda,\varrho})& = \lim_{T\to \infty}  \frac{1}{T} \int_{\Zz} \int_0^T \Phi(x,t) \overline{\chi_{\lambda,\varrho}(x,t)} dx dt\\
&=\int_{\Ss}\phi(z) \overline{\chi_q}(z) d\mu\\
&=\widehat{\phi}(q).
\end{align*}
where $q=\lambda+\varrho$. Finally, this allows us to `project' the Bohr -- Fourier series of any $\Z$ --invariant function $\Phi:\R\times\Zz\To \C$ to the classical Bohr -- Fourier series of a function 
$\phi:\Ss\to \C$ as:
\begin{align*}
\sphi(z)	&=
\sum_{q\in\Q} \widehat{\phi}(\chi_q) 
\chi_q(z).
\end{align*}

\end{document}